\documentclass[11pt]{amsart}

\usepackage{lmodern}
\usepackage{microtype}
\usepackage[english]{babel}
\usepackage{mathtools}
\usepackage{hyperref}
\usepackage[hyperpageref]{backref}
\usepackage[msc-links]{amsrefs}

\theoremstyle{plain} 
\newtheorem{theorem}{Theorem}[section]
\newtheorem{lemma}[theorem]{Lemma}
\newtheorem{corollary}[theorem]{Corollary}

\theoremstyle{remark}
\newtheorem*{remark}{Remark}

\numberwithin{equation}{section}

\DeclareMathOperator{\mre}{Re} 
\newcommand{\Q}{\mathcal{Q}}

\title[Contractivity of Rudin's diagonal restriction operator]{Contractivity of Rudin's diagonal restriction operator for $H^p$ of the polydisc}

\date{\today} 

\author[Brevig]{Ole Fredrik Brevig} 
\address{Department of Mathematical Sciences, Norwegian University of Science and Technology (NTNU), 7491 Trondheim, Norway} 
\email{ole.brevig@ntnu.no}
\author[G\'omez-Serrano]{Javier G\'omez-Serrano}\address{Department of Mathematics, Brown University, 02912 Providence, Rhode Island, USA} 
\email{javier\_gomez\_serrano@brown.edu}
\author[Ortega-Cerd\`{a}]{Joaquim Ortega-Cerd\`{a}}\address{Departament de Matem\`{a}tiques i Inform\`{a}tica, Universitat de Barcelona and Centre de Recerca Matem\`{a}tica, 08007 Barcelona, Spain}
\email{jortega@ub.edu}
\author[Seip]{Kristian Seip}\address{Department of Mathematical Sciences, Norwegian University of Science and Technology (NTNU), 7491 Trondheim, Norway} 
\email{kristian.seip@ntnu.no}

\begin{document} 

\begin{abstract}
	We establish that Rudin's diagonal restriction operator is contractive from the Hardy space $H^p$ of the bidisc to the Bergman space $A^p$ of the disc for $2\leq p<\infty$, and characterize all cases of equality. We extend the result to polydiscs and obtain applications to polynomial point evaluation among others. The proof relies on an explicit formula, deduced from certain Green's formulas, for the defect between the $p$th powers of the two $L^p$ norms.
\end{abstract}

\maketitle

\section{Introduction}
Let $m$ stand for the normalized Lebesgue measure on the unit circle $\mathbb{T}$ and $m_2 \coloneq  m \times m$ the corresponding product measure on $\mathbb{T}^2$. The Hardy space $H^p(\mathbb{D}^2)$ consists of analytic functions in the unit bidisc $\mathbb{D}^2$ that satisfy
\[\|f\|_{H^p(\mathbb{D}^2)}^p \coloneq \sup_{0<r<1} \int_{\mathbb{T}^2} \lvert f_r(z) \rvert^p\,dm_2(z) < \infty,\]
where $f_r$ denotes the dilation $f_r(z) \coloneq  f(rz_1,rz_2)$. The Bergman space $A^p(\mathbb{D})$ consists of the analytic functions in the unit disc $\mathbb{D}$ such that
\[\|g\|_{A^p(\mathbb{D})}^p \coloneq  \int_{\mathbb{D}} \lvert g(w) \rvert^p \,d\mu(w) < \infty,\]
where $\mu$ is the normalized Lebesgue area measure on $\mathbb{D}$. Let $Df(w) \coloneq f(w,w)$ denote the diagonal restriction operator. Our main result reads as follows.

\begin{theorem} \label{thm:d2}
	Suppose that $2<p<\infty$. We have
	\begin{equation} \label{eq:mainbound} \|Df\|_{A^p(\mathbb{D})} \leq \|f\|_{H^p(\mathbb{D}^2)}\end{equation}
	for every $f$ in $H^p(\mathbb{D}^2)$, and equality holds in this bound if and only if either $f \equiv 0$ or 
	\[f(z_1,z_2) = \left(\frac{a(z_1)-a(z_2)}{z_1-z_2}\right)^{2/p},\]
	where $a$ is a univalent function in $\mathbb{D}$ such that $a(\mathbb D)$ has finite area.
\end{theorem}

The same result for $p=2$ is essentially due to Rudin~\cite{Rudin1969}*{Exercise~3.4.4(e)}, although in this case the assumption on $a$ is merely that $a'$ belongs to $A^2(\mathbb{D})$. If $p=2k$ for a positive integer $k$, then the estimate $\|Df\|_{A^p(\mathbb{D})} \leq \|f\|_{H^p(\mathbb{D}^2)}$ follows directly from the case $p=2$ by the power trick. We will reprove Rudin's result using our methods (see Theorem~\ref{thm:rudin}).

Duren--Shields~\cite{DS1975} and Horowitz--Oberlin~\cite{HO1975} established independently that $D$ is bounded from $H^p(\mathbb{D}^2)$ to $A^p(\mathbb{D})$ for $0<p<\infty$. Theorem~\ref{thm:d2} refines this by showing that the norm of $D$ equals $1$ for $2< p<\infty$ and by identifying the extremal functions. The norm is no longer $1$ for $0<p<2$, since $f(z_1,z_2) = z_1+z_2$ satisfies
\[\frac{\|Df\|_{A^p(\mathbb{D})}^p}{\|f\|_{H^p(\mathbb{D}^2)}^p} = \frac{p\,\mathrm{B}(p/2,1/2)}{2+p},\]
which exceeds $1$ precisely when $0<p<2$ (see~\cite{BBHOCP2019}*{Lemma~21}).

Let us discuss the main idea in the proof of Theorem~\ref{thm:d2}. The motivation comes from the case $p=2$. Writing $f(z) = \sum_{\alpha} c_\alpha z^\alpha$ and following Rudin's proof, which applies the Cauchy--Schwarz inequality to the coefficients, we can express the defect $\|f\|_{H^2(\mathbb{D}^2)}^2 - \|Df\|_{A^2(\mathbb{D})}^2$ as a sum of nonnegative terms:
\begin{equation} \label{eq:2defect}
    \|f\|_{H^2(\mathbb{D}^2)}^2 = \|Df\|_{A^2(\mathbb{D})}^2 + \sum_{n=0}^{\infty} \left( \sum_{\lvert \alpha \rvert = n} \lvert c_\alpha \rvert^2
  - \frac{1}{n+1} \left\lvert \sum_{\lvert \alpha \rvert = n} c_\alpha \right\rvert^2 \right).
\end{equation}
It follows readily from \eqref{eq:2defect} that $\|f\|_{H^2(\mathbb{D}^2)} = \|Df\|_{A^2(\mathbb{D})}$ if and only if 
\begin{equation} \label{eq:p2ext}
    f(z_1,z_2) = \frac{a(z_1)-a(z_2)}{z_1-z_2},
\end{equation}
where $a'$ belongs to $A^2(\mathbb{D})$. Let $\partial_j$ and $\overline{\partial_j}$ stand for the Wirtinger derivatives with respect to the variables $z_j$ for $j=1,2$. The key object is the second-order differential operator 
\[\mathcal{L}u(z_1,z_2) \coloneq  \partial_1 \partial_2 \big((z_1-z_2) u(z_1,z_2)\big),\]
which annihilates $f$ of the form \eqref{eq:p2ext}. The idea is to express the defect term in \eqref{eq:2defect} as an integral of $\lvert\mathcal{L}f\rvert^2$ against a positive weight. It is this integral representation of the defect, rather than the inequality itself, that extends to $p>2$. Since
$\lvert \mathcal{L}f \rvert^2 = \mathcal{L} \overline{\mathcal{L}} \lvert f \rvert^2$ for
analytic functions $f$, we also introduce
$\Q \coloneq  \mathcal{L} \overline{\mathcal{L}} = \overline{\mathcal{L}}\mathcal{L}$. 

\begin{lemma} \label{lem:green2}
	If $u$ is $C^4$ on a neighborhood of the closed bidisc, then
	\[\int_{\mathbb{T}^2} u(z) \,dm_2(z) = \int_{\mathbb{D}} Du(w) \,d\mu(w) - 2\int_{\mathbb{D}^2} \log{\left\lvert\frac{z_1-z_2}{1-z_1 \overline{z_2}}\right\rvert} \Q u(z) \, d\mu(z_1)\,d\mu(z_2).\]
\end{lemma}

Notice that if we take $u=\lvert f\rvert^2$, then we can obtain \eqref{eq:2defect} from Lemma~\ref{lem:green2} by direct calculation. The proof of Lemma~\ref{lem:green2} amounts to two well-chosen applications of Green's formula. The presence of the Green's function of the disc
\[G(z_1,z_2) \coloneq \log{\left\lvert\frac{z_1-z_2}{1-z_1 \overline{z_2}}\right\rvert}\]
and the fact that $16 \Q u(z) = \Delta_1 \Delta_2 \big(\lvert z_1-z_2\rvert^2 u(z)\big)$ for the Laplacians $\Delta_j$ with respect to the variables $z_j$ for $j=1,2$ are strong hints in this direction.

In order to obtain Theorem~\ref{thm:d2} from Lemma~\ref{lem:green2}, the idea is to consider 
\[u = (\lvert f_r \rvert^2 + \varepsilon)^{p/2}\]
for $0<r<1$ and $\varepsilon>0$. We will show that the remainder term in Lemma~\ref{lem:green2} is nonnegative, before letting $\varepsilon \to 0^+$ and $r\to 1^-$. Here we will use both that $f$ is analytic and that $p/2\geq1$. If equality holds, then letting $\varepsilon \to 0^+$ and $r \to 1^-$ in the remainder term, we see that $\mathcal{L}(f^{p/2})$ vanishes in the open subset of $\mathbb{D}^2$ where $f$ does not vanish. A rigidity argument based on this shows that if $p>2$, then $f$ has no zeros in $\mathbb{D}^2$, and so $f^{p/2}$ is an extremal function for the case $p=2$. This is where the univalence of $a$ comes from, since \eqref{eq:p2ext} is nonvanishing if and only if $a$ is univalent.

\subsection*{Organization} We will prove Theorem~\ref{thm:d2} in Section~\ref{sec:proof}, and discuss its extension to the polydisc $\mathbb{D}^d$ in Section~\ref{sec:polydisc}. Some applications of our results are presented in Section~\ref{sec:applications}.

\subsection*{Acknowledgments} OFB was supported by Grant 354537 of the Research Council of Norway. JGS was partially supported by the National Science Foundation through Grants DMS-2247537, DMS-2434314, and DMS-2554957. JOC was supported by grant PID2024-160033NB-I00 of the Agencia Estatal de Investigación and by the 2024 ICREA 00142 grant of the Generalitat de Catalunya. KS was supported in part by Grant 334466 of the Research Council of Norway.

\subsection*{AI disclosure}
The AI work on this project began in June 2026, with the aim of proving \eqref{eq:mainbound}. A workflow combining GPT-5.6 Sol and Claude Fable~5 with regular input from the authors ran for about six weeks: it produced some new observations, but was mostly unsuccessful. In August 2026, using techniques similar to those in~\cite{BCOCS2024}, GPT-5.6 Sol established Corollary~\ref{cor:nikolskii}, which is equivalent to \eqref{eq:mainbound} for homogeneous polynomials. It also produced an incorrect proof of \eqref{eq:mainbound} based on decomposing a polynomial into its homogeneous parts. In September 2026, GPT-6 Astra proved \eqref{eq:mainbound} and \eqref{eq:maindbound} via Lemmas~\ref{lem:green2} and~\ref{lem:greend2}, and this proof forms the core of the present paper. It was then distilled into its current form through extensive and vigorous dialogue with Claude Fable~5.1 and Opus~5.5. The main contributions of this stage are simple proofs of Lemmas~\ref{lem:green2} and~\ref{lem:greend2} via Green's formula and the characterization of the equality cases in Theorems~\ref{thm:d2} and~\ref{thm:dgen}.

\section{Proof of \texorpdfstring{Theorem~\ref{thm:d2}}{Theorem 1.1}}
\label{sec:proof}
We begin with the proof of Lemma~\ref{lem:green2}. Recall that the Poisson kernel in $\mathbb{D}$ takes the form
\[P(w,\zeta) \coloneq \frac{1-\lvert w \rvert^2}{\lvert \zeta-w\rvert^2},\]
for $w$ in $\mathbb{D}$ and $\zeta$ on $\mathbb{T}$. As mentioned above, we need two versions of Green's formula in its proof.

\begin{lemma}[Green's formula] \label{lem:green1}
	Suppose that $\varphi$ is $C^2$ on a neighborhood of the closed unit disc.
	\begin{enumerate}
		\item[(a)] If $w$ is in $\mathbb{D}$, then
		\[\varphi(w) = \int_{\mathbb{T}} \varphi(\zeta)\, P(w,\zeta) \,dm(\zeta) + \frac{1}{2} \int_{\mathbb{D}} G(w,v)\, \Delta \varphi(v)\,d\mu(v).\]
		\item[(b)] If $\zeta$ is on $\mathbb{T}$, then
		\[\int_{\mathbb{T}} \varphi(\xi) \,dm(\xi) = \frac{1}{4} \int_{\mathbb{D}} P(w,\zeta)\, \Delta\big(\lvert w - \zeta \rvert^2 \varphi(w)\big) \,d\mu(w).\]
	\end{enumerate}
\end{lemma}

Part~(a) is the classical Green representation formula for the disc. Versions of~(b) can be found in the literature under additional assumptions. For $\varphi = \lvert g \rvert^2$ with $g$ analytic, it is the local Douglas formula of Richter and Sundberg~\cite{RS1991}*{Proposition~2.2} (see also~\cite{EKMR2014}*{Section~7.2}). A half-plane version for $\varphi = \lvert h \rvert^2$ with $h$ harmonic appears in~\cite{Garnett2007}*{Chapter~VI, Exercise~6}. For the convenience of the reader, we include a short proof.

\begin{proof}[Proof of Lemma~\ref{lem:green1}~\textup{(b)}]
	For $\varepsilon>0$, let $\mathbb{D}_\varepsilon \coloneq \{w \in \mathbb{D} : \lvert w-\zeta \rvert > \varepsilon\}$ and consider
	\[\int_{\mathbb{D}_\varepsilon} P(w,\zeta)\,\Delta\big(\lvert w-\zeta\rvert^2\varphi(w)\big)\,d\mu(w).\]
	Since $P(\cdot,\zeta)$ is harmonic in $\mathbb{D}$, Green's second identity expresses this integral as an integral over $\partial \mathbb{D}_\varepsilon$. On $\mathbb{T}\setminus\{\zeta\}$, the Poisson kernel vanishes and its outward normal derivative is $-2/\lvert w-\zeta\rvert^2$, so this part of $\partial\mathbb{D}_\varepsilon$ contributes
	\[4\int_{\lvert \xi-\zeta\rvert>\varepsilon} \varphi(\xi)\,dm(\xi).\]
	The remaining part of $\partial \mathbb{D}_\varepsilon$ is an arc of length at most $\pi\varepsilon$, and standard estimates show that its contribution tends to $0$ as $\varepsilon\to0^+$, owing to the factor $\lvert w-\zeta\rvert^2$. Letting $\varepsilon\to0^+$, we obtain~(b).
\end{proof}

\begin{proof}[Proof of Lemma~\ref{lem:green2}]
	Set $U(z) \coloneq \lvert z_1 - z_2 \rvert^2 u(z)$, so that $16 \Q u = \Delta_1 \Delta_2 U$. We begin by using Lemma~\ref{lem:green1}~(b) with $\varphi = u(\cdot,\zeta_2)$ and $\zeta=\zeta_2$, followed by Fubini's theorem to obtain
	\[\int_{\mathbb{T}}\int_\mathbb{T} u(\zeta_1,\zeta_2)\,dm(\zeta_1)\,dm(\zeta_2) = \frac{1}{4} \int_{\mathbb{D}} \int_{\mathbb{T}} \Delta_1 U(z_1,\zeta_2) P(z_1,\zeta_2)\,dm(\zeta_2)\,d\mu(z_1).\]
	The function $\Delta_1 U$ appearing in the integrand to the right is
	\[\Delta_1 U(z) = 4\big(u(z) + (z_1-z_2) \partial_1 u(z) + (\overline{z_1}-\overline{z_2}) \overline{\partial_1} u(z)\big) + \lvert z_1 - z_2 \rvert^2 \Delta_1 u(z),\]
	so that $\Delta_1 U(z_1,z_1) = 4Du(z_1)$. Using Lemma~\ref{lem:green1}~(a) with $\varphi = \Delta_1 U(z_1,\cdot)$ and $w=z_1$, we therefore find that 
	\[\frac{1}{4}\int_{\mathbb{T}} \Delta_1 U(z_1,\zeta_2) P(z_1,\zeta_2)\,dm(\zeta_2) = Du(z_1) - 2\int_{\mathbb{D}} G(z_1,z_2)\, \Q u(z) \,d\mu(z_2).\]
	The proof is completed by combining these two expressions.
\end{proof}

Before continuing with the proof of Theorem~\ref{thm:d2}, we will illustrate its use by establishing the corresponding result for $p=2$ from Lemma~\ref{lem:green2}. 

\begin{theorem}[Rudin~\cite{Rudin1969}*{Exercise~3.4.4(e)}] \label{thm:rudin}
	We have
	\[\|Df\|_{A^2(\mathbb{D})} \leq \|f\|_{H^2(\mathbb{D}^2)}\]
	for every $f$ in $H^2(\mathbb{D}^2)$, and equality holds in this bound if and
	only if
	\[f(z_1,z_2) = \frac{a(z_1)-a(z_2)}{z_1-z_2}\]
	for some analytic function $a$ in $\mathbb{D}$ such that $a'$ belongs to
	$A^2(\mathbb{D})$.
\end{theorem}

\begin{proof}
	Let $f$ be in $H^2(\mathbb{D}^2)$ and let $0<r<1$. Since $\Q \lvert f_r \rvert^2 = \lvert \mathcal{L}f_r \rvert^2$ and $G<0$ in $\mathbb{D}^2$, we can apply Lemma~\ref{lem:green2} with $u = \lvert f_r \rvert^2$ and let $r \to 1^-$ to obtain the stated inequality.

	Since $\mathcal{L}f_r(z) = r\mathcal{L}f(rz)$, Fatou's lemma shows that equality holds if and only if $\mathcal{L}f\equiv0$ in $\mathbb{D}^2$. Solving this equation, we find that
	\[(z_1-z_2) f(z) = a(z_1)+b(z_2),\]
	for analytic functions $a$ and $b$ in $\mathbb{D}$, and setting $z_1=z_2$, we get $a = -b$. Since $Df = a'$ and $\|f\|_{H^2(\mathbb{D}^2)} = \|a'\|_{A^2(\mathbb{D})}$, we find that $a'$ is in $A^2(\mathbb{D})$.
\end{proof}

In connection with the upcoming proof of Theorem~\ref{thm:d2}, we note that the extremal in Theorem~\ref{thm:rudin} does not vanish in $\mathbb{D}^2$ if and only if $a$ is univalent, in which case $a(\mathbb{D})$ has area $\pi\|a'\|_{A^2(\mathbb{D})}^2$. The relevance of this is that if we can establish that the extremals in Theorem~\ref{thm:d2} do not vanish in $\mathbb{D}^2$, then we can use Theorem~\ref{thm:rudin} to describe them. 

The proof of Theorem~\ref{thm:d2} requires two additional results. The first replaces $\Q \lvert f_r \rvert^2 = \lvert \mathcal{L}f_r \rvert^2 \geq 0$ from the proof of Theorem~\ref{thm:rudin}, while the second replaces the assertion that $f$ is extremal if and only if $\mathcal{L}f\equiv0$ from the proof of Theorem~\ref{thm:rudin}. The proof of the former is reduced to a long calculation, while the latter requires some finesse.

\begin{lemma} \label{lem:positivity}
	Fix $2<p<\infty$ and suppose that $f$ is analytic on a neighborhood of the
	closed bidisc. If $\varepsilon>0$, then
	\[\Q \big(\lvert f \rvert^2+\varepsilon\big)^{p/2} \geq 0.\]
\end{lemma}

\begin{proof}
	Set $q \coloneq p/2$, $\delta(z) \coloneq z_1-z_2$, $\varphi(x) \coloneq (x+\varepsilon)^q$ for $x\geq0$, and $u\coloneq\varphi(\lvert f \rvert^2)$. Since $f$ is analytic, we have $\partial_j \lvert f \rvert^2 = \overline{f}\,\partial_j f$ for $j=1,2$. Inserting this into $\mathcal{L}u = \delta\,\partial_1\partial_2 u - \partial_1 u + \partial_2 u$, we obtain 
	\begin{equation} \label{eq:chainrule}
		\mathcal{L}u = \varphi'\,\overline{f}\,\mathcal{L}f + \varphi''\,\overline{f}^2\,\delta\,\partial_1 f\,\partial_2 f.
	\end{equation}
	Here and below, functions of $x$ are evaluated at $x=\lvert f \rvert^2$ whenever they appear together with $f$. Since $\mathcal{L}f$ and $\delta\,\partial_1 f\,\partial_2 f$ are analytic, they behave as constants under $\overline{\mathcal{L}}$, so it remains to compute $\overline{\mathcal{L}}$ of $\varphi'\,\overline{f}$ and $\varphi''\,\overline{f}^2$. For this purpose, we introduce
	\[A_\varepsilon(x) \coloneq \left(x\varphi'(x)\right)', \quad B_\varepsilon(x) \coloneq A_\varepsilon'(x), \quad\text{and}\quad C_\varepsilon(x) \coloneq \left(xB_\varepsilon(x)\right)'.\]
	Using that $\overline{\partial_j}\lvert f \rvert^2 = f\,\overline{\partial_j f}$ for $j=1,2$, we find that
	\[\overline{\mathcal{L}}\big(\varphi'\,\overline{f}\big) = A_\varepsilon\,\overline{\mathcal{L}f} + fB_\varepsilon\,\overline{\delta\,\partial_1 f\,\partial_2 f}\quad\text{and}\quad
\overline{\mathcal{L}}\big(\varphi''\,\overline{f}^2\big) = \overline{f}B_\varepsilon\,\overline{\mathcal{L}f} + C_\varepsilon\,\overline{\delta\,\partial_1 f\,\partial_2 f}.\]
	Applying $\overline{\mathcal{L}}$ to \eqref{eq:chainrule}, we therefore obtain
	\begin{equation} \label{eq:quadratic}
		\Q u = A_\varepsilon \lvert \mathcal{L}f \rvert^2
		+ 2\mre(f B_\varepsilon\, \mathcal{L}f\,
		\overline{\delta\, \partial_1 f\, \partial_2 f})
		+ C_\varepsilon \lvert \delta\, \partial_1 f\, \partial_2 f\rvert^2.
	\end{equation}
	Since $q>1$, both $\varphi'$ and $\varphi''$ are positive, so $A_\varepsilon = \varphi'+x\varphi''>0$. Moreover, a direct computation shows that
	\[A_\varepsilon(x) C_\varepsilon(x) - x B_\varepsilon(x)^2 = q^2(q-1)\varepsilon(x+\varepsilon)^{2q-6}\left(2\varepsilon^2+2q\varepsilon x+q(q-1)x^2\right) \geq 0.\]
	Hence, completing the square in \eqref{eq:quadratic}, we get
	\[\Q u = A_\varepsilon  \left\lvert \mathcal{L}f
	+ \frac{\overline{f}B_\varepsilon }{A_\varepsilon }\, \delta\, \partial_1 f\, \partial_2 f
	\right\rvert^2 	
	+ \frac{A_\varepsilon C_\varepsilon  - \lvert f \rvert^2 B_\varepsilon^2}{A_\varepsilon }
	\lvert \delta\, \partial_1 f\, \partial_2 f\rvert^2 \geq 0. \qedhere\]
\end{proof}

The finesse is supplied by the following elementary one-variable statement.

\begin{lemma} \label{lem:onevar}
	Fix $\kappa>0$. If $\varphi$ and $\psi$ are analytic in $\mathbb{D}$ and
	\[\varphi \psi' + \kappa \varphi' \psi \equiv 0,\]
	then either $\varphi\psi \equiv 0$ or $\varphi\psi$ does not vanish in $\mathbb{D}$.
\end{lemma}

\begin{proof}
	We may assume that $\varphi\psi\not\equiv0$. Near any point where $\varphi\psi \neq 0$, a branch of $\varphi^\kappa$ is analytic and
	\[(\varphi^\kappa \psi)' = \varphi^{\kappa-1}(\kappa \varphi'\psi + \varphi \psi') \equiv 0\]
	by assumption. Hence $\lvert \varphi \rvert^\kappa \lvert \psi \rvert$ is locally constant in the complement of the discrete set $\{\varphi\psi=0\}$, which is connected, so it is constant there. Since $\kappa>0$, the function $\lvert \varphi \rvert^\kappa \lvert \psi \rvert$ is continuous in $\mathbb{D}$ and therefore constant. This constant is positive since $\varphi\psi\not\equiv0$ by assumption.
\end{proof}

\begin{lemma} \label{lem:rigidity}
	Fix $q>1$. If $f$ is analytic in $\mathbb{D}^2$ and 
    \[\mathcal{L}(f^q) \equiv 0\] 
    in $\{f\neq0\}$ for every local branch of $f^q$, then either $Df \equiv 0$ or $f$ does not vanish in $\mathbb{D}^2$.
\end{lemma}

\begin{proof}
	Assume that $Df \not\equiv 0$ and set
	\[g(z) \coloneq  f(z) + q(z_1-z_2)\,\partial_1 f(z).\]
	In $\{f\neq0\}$, a direct computation shows that $\partial_1\big((z_1-z_2) f^q\big)= f^{q-1} g $ for any local branch of $f^q$. It follows that
	\[\mathcal{L}(f^q) = \partial_2 \big(f^{q-1} g\big) = f^{q-2} \big(f\, \partial_2 g + (q-1)\, g\, \partial_2 f\big).\]
	By assumption, the analytic function $f\, \partial_2 g + (q-1)\, g\, \partial_2 f$ therefore vanishes in the nonempty open set $\{f\neq0\}$, and hence identically in $\mathbb{D}^2$. 
    
    For fixed $z_1$ in $\mathbb{D}$, we may consequently apply Lemma~\ref{lem:onevar} with $\kappa = q-1$ to $\varphi(w) = f(z_1,w)$ and $\psi(w) = g(z_1,w)$. If $f(z)=0$, then $\varphi\psi$ vanishes at $w=z_2$, so Lemma~\ref{lem:onevar} shows that $\varphi\psi\equiv0$, and in particular 
    \[(Df(z_1))^2 = \varphi(z_1)\psi(z_1) = 0.\]
    As $\mathcal{L}(f^q)$ merely changes sign when the variables $z_1$ and $z_2$ are interchanged, the same argument shows that if $f(z)=0$, then $Df(z_2)=0$. The zero set of $Df$ is discrete since $Df\not\equiv0$, so it follows that the zero set of $f$ is discrete. However, a nontrivial analytic function in $\mathbb{D}^2$ has no isolated zeros (see e.g.~\cite{Rudin1969}*{Section~1.2.2}). Hence it follows that $f$ does not vanish in $\mathbb{D}^2$.
\end{proof}

\begin{proof}[Proof of Theorem~\ref{thm:d2}]
	Let $f$ be in $H^p(\mathbb{D}^2)$, $0<r<1$, and $\varepsilon>0$. Using Lemmas~\ref{lem:green2} and~\ref{lem:positivity} and letting $\varepsilon\to0^+$, we get  
	\[\|f_r\|_{H^p(\mathbb{D}^2)} \geq \|Df_r\|_{A^p(\mathbb{D})}.\]
	Letting $r \to 1^-$, we obtain \eqref{eq:mainbound}. Equality in \eqref{eq:mainbound} holds if and only if 	
	\[\lim_{r \to 1^-} \lim_{\varepsilon \to 0^+} \int_{\mathbb{D}^2} G(z_1,z_2)\,\Q\, (\lvert f_r(z) \rvert^2+\varepsilon)^{p/2}\,d\mu(z_1)\,d\mu(z_2) = 0.\]
	Suppose that this is the case and that $f \not\equiv 0$. Fix a point $z$ where $f_r(z)\neq0$ and let $F_r$ be a local branch of $f_r^{p/2}$ near $z$. Since $(\lvert f_r \rvert^2+\varepsilon)^{p/2}$ depends smoothly on $\varepsilon\geq0$ near $z$, we find that
	\[\lim_{\varepsilon \to0^+} \Q (\lvert f_r(z) \rvert^2+\varepsilon)^{p/2} = \Q \lvert F_r(z)\rvert^2 = \lvert \mathcal{L}F_r(z)\rvert^2,\]
	where the right-hand side does not depend on the choice of branch and will be denoted $\lvert \mathcal{L}(f_r^{p/2})(z)\rvert^2$. Since $G \leq 0$, two applications of Fatou's lemma therefore give that
	\[\int_{\{f\neq0\}} G(z_1,z_2) \, \lvert \mathcal{L}(f^{p/2})(z)\rvert^2\, d\mu(z_1)\,d\mu(z_2)=0.\]
	Since $G<0$ almost everywhere, it follows that $\mathcal{L}(f^{p/2})=0$ almost everywhere, and hence everywhere, in the open set $\{f\neq0\}$. As $\|Df\|_{A^p(\mathbb{D})} = \|f\|_{H^p(\mathbb{D}^2)} \neq 0$, Lemma~\ref{lem:rigidity} with $q=p/2$ shows that $f$ does not vanish in $\mathbb{D}^2$. It follows that $f^{p/2}$ is a well-defined analytic function in $\mathbb{D}^2$, which belongs to $H^2(\mathbb{D}^2)$ because $\lvert f^{p/2} \rvert^2 = \lvert f \rvert^p$. Consequently, $f$ attains equality in \eqref{eq:mainbound} if and only if $f^{p/2}$ attains equality in Theorem~\ref{thm:rudin}. The final assertion therefore follows from Theorem~\ref{thm:rudin} and the discussion following its proof.
\end{proof}

\section{Extension to the polydisc} \label{sec:polydisc}
The Hardy space $H^p(\mathbb{D}^d)$ is defined as above, with $m_d \coloneq  m\times\cdots\times m$ in place of $m_2$. For an integer $d\geq2$, let $\mu_d$ be the probability measure on $\mathbb{D}$ given by
\[d\mu_d(w) \coloneq  (d-1)(1-\lvert w \rvert^2)^{d-2}\,d\mu(w).\]
The weighted Bergman space $A^p_d(\mathbb{D})$ consists of the analytic functions $g$ in $\mathbb{D}$ such that
\[\|g\|_{A^p_d(\mathbb{D})}^p \coloneq  \int_{\mathbb{D}} \lvert g(w) \rvert^p\,d\mu_d(w) < \infty,\]
so that $A^p_2(\mathbb{D}) = A^p(\mathbb{D})$. We write $D_d f(w) \coloneq f(w,\ldots,w)$ for the diagonal restriction operator so that $D_2 = D$.

Let $a$ be an analytic function in $\mathbb{D}$ and $z$ a point in $\mathbb{D}^d$. The order $d-1$ divided difference of $a$ is defined as
\begin{equation} \label{eq:dd}
    a[z_1,\ldots,z_d] \coloneq  \frac{1}{2\pi i}\int_{\lvert w \rvert = \rho} \frac{a(w)}{(w-z_1)\cdots(w-z_d)}\,dw,
\end{equation}
where $\|z\|_{\ell^\infty} < \rho < 1$. It is a symmetric analytic function in $\mathbb{D}^d$ that satisfies $a[w,\ldots,w] = a^{(d-1)}(w)/(d-1)!$, and for $d=2$ we have
\[a[z_1,z_2] = \frac{a(z_1)-a(z_2)}{z_1-z_2}.\]
We begin with the case $p=2$, where we follow Rudin's original coefficient-based argument.

\begin{theorem} \label{thm:rudind}
	Suppose that $d\geq2$. We have
	\[\|D_d f\|_{A^2_d(\mathbb{D})} \leq \|f\|_{H^2(\mathbb{D}^d)}\]
	for every $f$ in $H^2(\mathbb{D}^d)$, and equality holds in this bound if and only if
	\[f(z) = a[z_1,\ldots,z_d]\]
	for some analytic function $a$ in $\mathbb{D}$ such that $a^{(d-1)}$ belongs to $A^2_d(\mathbb{D})$.
\end{theorem}

\begin{proof}
    The key observation is the identity
    \begin{equation} \label{eq:binom}
    \int_\mathbb{D} \lvert w \rvert^{2s}\,d\mu_d(w) = \frac{\Gamma(d)\,\Gamma(1+s)}{\Gamma(d+s)},
    \end{equation}
	valid for $s>-1$, which for $s=n$ equals $1/\binom{n+d-1}{n}$.
	Let $f(z) = \sum_\alpha c_\alpha z^\alpha$ be in $H^2(\mathbb{D}^d)$. By orthogonality and \eqref{eq:binom}, we have
	\[\|f\|_{H^2(\mathbb{D}^d)}^2 = \|D_d f\|_{A^2_d(\mathbb{D})}^2 + \sum_{n=0}^\infty \left(\sum_{\lvert \alpha \rvert = n} \lvert c_\alpha \rvert^2 - \frac{1}{\binom{n+d-1}{n}}\left\lvert \sum_{\lvert \alpha \rvert = n} c_\alpha \right\rvert^2\right).\]
	Since $\binom{n+d-1}{n}$ is also the number of multi-indices $\alpha$ in $\mathbb{N}_0^d$ with $\lvert \alpha \rvert = n$, the Cauchy--Schwarz inequality shows that every term in the infinite sum is nonnegative and vanishes if and only if $c_\alpha$ depends only on $\lvert \alpha \rvert$, say $c_\alpha = b_n$ for $\lvert \alpha \rvert = n$. Since $\sum_{\lvert \alpha \rvert = n} z^\alpha$ is the divided difference of $w\mapsto w^{n+d-1}$ by \eqref{eq:dd}, equality holds if and only if $f(z) = a[z_1,\ldots,z_d]$ with $a(w) = \sum_n b_n w^{n+d-1}$. Finally, for such $f$ we have $D_d f = a^{(d-1)}/(d-1)!$ and
	\[\|f\|_{H^2(\mathbb{D}^d)}^2 = \sum_{n=0}^\infty \binom{n+d-1}{n}\lvert b_n \rvert^2 = \|D_d f\|_{A^2_d(\mathbb{D})}^2\]
	by \eqref{eq:binom}. Hence $f$ belongs to $H^2(\mathbb{D}^d)$ if and only if $a^{(d-1)}$ is in $A^2_d(\mathbb{D})$.
\end{proof}

As in the case $d=2$, the description of the functions attaining equality in Theorem~\ref{thm:rudind} will be needed when we turn to $p>2$.

\begin{theorem} \label{thm:dgen}
	Suppose that $d\geq2$ and $2<p<\infty$. We have
	\begin{equation} \label{eq:maindbound}\|D_d f\|_{A^p_d(\mathbb{D})} \leq \|f\|_{H^p(\mathbb{D}^d)}\end{equation}
	for every $f$ in $H^p(\mathbb{D}^d)$, and equality holds in this bound if and only if either $f\equiv0$ or
	\[f(z) = \left(a[z_1,\ldots,z_d]\right)^{2/p},\]
	where $a$ is an analytic function in $\mathbb{D}$ such that $a^{(d-1)}$ belongs to $A^2_d(\mathbb{D})$ and such that $a[z_1,\ldots,z_d]$ does not vanish in $\mathbb{D}^d$.
\end{theorem}

\begin{remark}
    For $d=2$, the requirement that $a[z_1,z_2]$ does not vanish in $\mathbb{D}^2$ means that $a$ is univalent. In general, $a[z_1,\ldots,z_d]$ is the leading coefficient of the polynomial of degree at most $d-1$ that interpolates $a$ at $z_1,\ldots,z_d$, in the sense of Hermite interpolation if some of the points coincide. Consequently, $a[z_1,\ldots,z_d]$ does not vanish in $\mathbb{D}^d$ if and only if $a-P$ has at most $d-1$ zeros in $\mathbb{D}$, counting multiplicity, for every polynomial $P$ of degree at most $d-2$.
\end{remark}

The proof of Theorem~\ref{thm:dgen} follows the proof of Theorem~\ref{thm:d2}, which is the case $d=2$, by restricting one variable at a time to the diagonal. The first restriction is handled by Theorem~\ref{thm:d2}, applied to two-variable slices. In the $k$th restriction, the $k-1$ variables already restricted are integrated against $\mu_{k-1}$ instead of over $\mathbb{T}$, so we need the following substitute for Lemma~\ref{lem:green1}~(b). 

\begin{lemma} \label{lem:greend1}
	Suppose that $k\geq3$ and that $\varphi$ is $C^2$ on a neighborhood of the closed unit disc. If $\zeta$ is on $\mathbb{T}$, then
	\[\int_{\mathbb{D}} \varphi(w)\,d\mu_{k-1}(w) = \frac{1}{4(k-1)} \int_{\mathbb{D}} P(w,\zeta)^{k-1}\,\Delta(\lvert w-\zeta \rvert^{2k-2}\varphi(w))\,d\mu(w).\]
\end{lemma}
\begin{proof} This follows from Green's second identity applied to $P(w,\zeta)^{k-1}$ and $\lvert w-\zeta \rvert^{2k-2}\varphi(w)$ as functions of $w$, since the boundary terms vanish for $k\geq3$ and a direct computation gives
\[\lvert w-\zeta \rvert^{2k-2}\,\Delta(P(w,\zeta)^{k-1})\,d\mu(w) = 4(k-1)\,d\mu_{k-1}(w). \qedhere \]
\end{proof}

For $k\geq2$, let
\begin{equation} \label{eq:Lkdef}
    \mathcal{L}_k u(z) \coloneq \partial_2 ((z_1-z_2)^{2-k}\,\partial_1 ((z_1-z_2)^{k-1} u(z)))
\end{equation}
and $\Q_k \coloneq \mathcal{L}_k\overline{\mathcal{L}_k} = \overline{\mathcal{L}_k}\mathcal{L}_k$, so that $\mathcal{L}_2 = \mathcal{L}$ and $\Q_2 = \Q$. Since holomorphic derivatives commute with multiplication by antiholomorphic functions and vice versa, we have
\[16\,\Q_k u(z) = \Delta_2(\lvert z_1-z_2 \rvert^{4-2k}\,\Delta_1(\lvert z_1-z_2 \rvert^{2k-2}u(z))).\]
Following the proof of Lemma~\ref{lem:green2}, but with Lemma~\ref{lem:green1}~(b) replaced by Lemma~\ref{lem:greend1}, we arrive at the following result. 

\begin{lemma} \label{lem:greend2}
	Suppose that $k\geq3$. If $u$ is $C^4$ on a neighborhood of the closed bidisc, then
	\begin{multline*}
        \int_{\mathbb{T}}\int_{\mathbb{D}} u(w,\zeta)\,d\mu_{k-1}(w)\,dm(\zeta) \\ = \int_{\mathbb{D}} Du(w)\,d\mu_{k}(w) - \frac{2}{(k-1)^2}\int_{\mathbb{D}}\int_{\mathbb{D}} G(z_1,z_2)\,\Q_{k} u(z)\,d\mu(z_2)\,d\mu_{k}(z_1).
    \end{multline*}
\end{lemma}
The only new ingredients in the proof of this lemma are the identity
\[(k-1)P(z_1,\zeta_2)^{k-1}\,d\mu(z_1) = P(z_1,\zeta_2)\,\lvert z_1-\zeta_2 \rvert^{4-2k}\,d\mu_{k}(z_1)\]
and the fact that $\lvert z_1-z_2 \rvert^{4-2k}\,\Delta_1(\lvert z_1-z_2 \rvert^{2k-2}u(z))$ equals $4(k-1)^2 Du(z_1)$ when $z_2=z_1$. We omit the details.

The proof of Lemma~\ref{lem:positivity} only uses that $\mathcal{L}$ is of the form
\[\mathcal{L}u(z) = (z_1-z_2)\,\partial_1\partial_2 u(z) + c_1\,\partial_1 u(z) + c_2\,\partial_2 u(z)\] for some constants $c_1$ and $c_2$. Expanding \eqref{eq:Lkdef}, we see that $\mathcal{L}_{k}$ is of this form with $c_1=-1$ and $c_2 = k-1$, so the same proof yields the following result.

\begin{lemma} \label{lem:positivityd}
	Fix $2<p<\infty$ and $k\geq3$, and suppose that $f$ is analytic on a neighborhood of the closed bidisc. If $\varepsilon>0$, then
	\[\Q_{k} \big(\lvert f \rvert^2+\varepsilon\big)^{p/2} \geq 0.\]
\end{lemma}

\begin{proof}[Proof of Theorem~\ref{thm:dgen}]
	The case $d=2$ is Theorem~\ref{thm:d2}, so we suppose that $d\geq3$. Let $f$ be in $H^p(\mathbb{D}^d)$ and fix $0<r<1$. We set $I_1 \coloneq  \|f_r\|_{H^p(\mathbb{D}^d)}^p$ and
	\[I_k \coloneq  \int_{\mathbb{T}^{d-k}}\int_{\mathbb{D}} \lvert f_r(w,\ldots,w,\zeta) \rvert^p\,d\mu_k(w)\,dm_{d-k}(\zeta)\]
	for $2\leq k\leq d$, so that in particular $I_d = \|D_d f_r\|_{A^p_d(\mathbb{D})}^p$. We claim that $I_k\leq I_{k-1}$ for $2\leq k\leq d$. For $k=2$, this follows from Theorem~\ref{thm:d2} applied to $f_r(\cdot,\cdot,\zeta)$ for $\zeta$ on $\mathbb{T}^{d-2}$ and Fubini's theorem. For $3\leq k\leq d$, we fix $\zeta$ on $\mathbb{T}^{d-k}$ and use Lemma~\ref{lem:greend2} with
	\[u(z) = (\lvert f_r(z_1,\ldots,z_1,z_2,\zeta) \rvert^2+\varepsilon)^{p/2}\]
    for $\varepsilon>0$. Since $\Q_k u\geq0$ by Lemma~\ref{lem:positivityd} and $G\leq0$, letting $\varepsilon\to0^+$ and integrating $\zeta$ over $\mathbb{T}^{d-k}$, we obtain the claim. 	Hence
	\begin{equation} \label{eq:chain}
		\|D_d f_r\|_{A^p_d(\mathbb{D})}^p = I_d \leq I_{d-1} \leq \cdots \leq I_2 \leq I_1 = \|f_r\|_{H^p(\mathbb{D}^d)}^p,
	\end{equation}
	and letting $r\to1^-$, we arrive at \eqref{eq:maindbound}.

	Suppose next that equality holds in \eqref{eq:maindbound} and that $f\not\equiv0$, and set $q\coloneq p/2$. By \eqref{eq:chain}, we have $I_2\to\|f\|_{H^p(\mathbb{D}^d)}^p$ as $r\to1^-$. For $\zeta$ on $\mathbb{T}^{d-2}$, let $f_\zeta(z_1,z_2) \coloneq  f(z_1,z_2,\zeta)$, understood in terms of boundary values. Arguing as in the case $k=2$ above, but with subharmonicity, first in $\zeta$ and then in $w$, to pass from $f_r$ to $f$, we find that
	\[I_2 \leq \int_{\mathbb{T}^{d-2}} \|Df_\zeta\|_{A^p(\mathbb{D})}^p\,dm_{d-2}(\zeta) \leq \int_{\mathbb{T}^{d-2}} \|f_\zeta\|_{H^p(\mathbb{D}^2)}^p\,dm_{d-2}(\zeta) = \|f\|_{H^p(\mathbb{D}^d)}^p.\]
	Letting $r\to1^-$, we see that $f_\zeta$ attains equality in \eqref{eq:mainbound} of Theorem~\ref{thm:d2} for almost every $\zeta$ on $\mathbb{T}^{d-2}$. 
    Hence either $f_\zeta\equiv0$, or $f_\zeta$ does not vanish and $\mathcal{L}(f_\zeta^q)\equiv0$. Consider the analytic function
    \[h(z) \coloneq f(z)\,\mathcal{L}f(z) + (q-1)(z_1-z_2)\,\partial_1 f(z)\,\partial_2 f(z)\]
    in $\mathbb{D}^d$, with $\mathcal{L}$ acting on the variables $z_1$ and $z_2$. Let $h_\zeta$ be defined from $f_\zeta$ in the same way. Since $\mathcal{L}(f_\zeta^q) = q f_\zeta^{q-2} h_\zeta$, it follows that $h_\zeta$ vanishes identically for almost every $\zeta$ on $\mathbb{T}^{d-2}$. As $h_\zeta$ is given by the boundary values of $h$ in the last $d-2$ variables, we conclude that $h\equiv0$ in $\mathbb{D}^d$.

    Consequently, for every $\xi$ in $\mathbb{D}^{d-2}$, the function $(z_1,z_2)\mapsto f(z_1,z_2,\xi)$ satisfies the assumption of Lemma~\ref{lem:rigidity}. Hence, if $f(z)=0$, then 
    \[f(w,w,z_3,\ldots,z_d)=0\]
    for every $w$ in $\mathbb{D}$. The norms in Theorem~\ref{thm:dgen} are invariant under permutations of the variables, so the same holds with $z_1,z_2$ replaced by any two of the variables. Applying this successively to the pairs $(z_2,z_3),\ldots,(z_{d-1},z_d)$, each time setting both variables of the pair equal to $w$, we find that if $f$ vanishes somewhere in $\mathbb{D}^d$, then $D_d f\equiv0$. The proof is finished in the same way as the proof of Theorem~\ref{thm:d2}, by using Theorem~\ref{thm:rudind}.
\end{proof}

\section{Applications}\label{sec:applications}

\subsection{Point estimates for polynomials}
Theorem~\ref{thm:d2} has the following consequence for point evaluation of polynomials on the unit circle.

\begin{corollary} \label{cor:nikolskii}
	Suppose that $2 < p<\infty$. If $P(w) = \sum_{k=0}^{n} c_k w^{k}$, then
	\begin{equation} \label{eq:nikolskii}
		\|P\|_{L^\infty(\mathbb{T})} \leq \left(\frac{np}{2}+1\right)^{1/p} \|P\|_{L^p(\mathbb{T})}.
	\end{equation}
\end{corollary}

\begin{proof}
	Suppose that $\|P\|_{L^\infty(\mathbb{T})}=P(1)$. The $n$-homogeneous polynomial
	\[f(z_1,z_2) = \sum_{k=0}^{n} c_k\, z_1^{k} z_2^{n-k}\]
	satisfies $\lvert f(e^{i\theta_1},e^{i\theta_2})\rvert = \lvert P(e^{i(\theta_1-\theta_2)})\rvert$ and $Df(w) = P(1) w^{n}$, which implies that $\|f\|_{H^p(\mathbb{D}^2)} = \|P\|_{L^p(\mathbb{T})}$ and $\|Df\|_{A^p(\mathbb{D})} = \|P\|_{L^\infty(\mathbb{T})} (np/2+1)^{-1/p}$. Apply Theorem~\ref{thm:d2} to complete the proof.
\end{proof}

Corollary~\ref{cor:nikolskii} confirms a conjecture due to Instanes, who established \eqref{eq:nikolskii} for $2 \leq p \leq 4$ and $n\leq 4$ in~\cite{Instanes}*{Theorem~1}. Note that for $p\geq6.8$, the conjecture can be settled by much simpler means~\cite{Instanes}*{Theorem~2}, so the difficulty lies in the range of small $p>2$.

\begin{remark}
	The combination of Corollary~\ref{cor:nikolskii} with a result due to Levin and Lubinsky~\cite{LL2015}*{Theorem~2.1} shows that if $f$ is an entire function of exponential type at most $\pi$ whose restriction to $\mathbb{R}$ is in $L^p(\mathbb{R})$ for some $2<p<\infty$, then
	\[\|f\|_{L^\infty(\mathbb{R})} \leq \left(\frac{p}{2}\right)^{1/p} \|f\|_{L^p(\mathbb{R})}.\]
	A strict version of this estimate was established in~\cite{BCOCS2024}*{Theorem~1.1} by different techniques. On an autobiographical note, the work in~\cite{BCOCS2024} began after the intersection of its author list with that of the present paper had tried (and failed) to establish Theorem~\ref{thm:d2} and Corollary~\ref{cor:nikolskii}. In view of subsequent developments~\cite{BORS2026}, this failure may in hindsight be regarded as serendipitous.
\end{remark}

In the classical literature, such estimates are usually formulated in the equivalent setting of trigonometric polynomials. If $T(e^{i\theta}) = \sum_{k=-n}^{n} c_k e^{ik\theta}$, then $e^{in\theta}\,T(e^{i\theta})$ is an analytic polynomial of degree $2n$, so Corollary~\ref{cor:nikolskii} yields
\[\|T\|_{L^\infty(\mathbb{T})} \leq (np+1)^{1/p}\,\|T\|_{L^p(\mathbb{T})}.\]
Estimates of this type go back at least to Jackson~\cite{Jackson1933}. The power trick (see Ibragimov~\cite{Ibragimov1958} and Timan~\cite{Timan1994}*{Section~4.9.2}) yields
\[\|T\|_{L^\infty(\mathbb{T})} \leq \left(2n\left\lceil \frac{p}{2} \right\rceil+1\right)^{1/p}\,\|T\|_{L^p(\mathbb{T})},\]
and the two estimates coincide when $p$ is an even integer.

More generally, Nessel and Wilmes~\cite{NW1978}*{Theorem~2} showed that if $T$ is a trigonometric polynomial in $d$ variables, then 
\[\|T\|_{L^\infty(\mathbb{T}^d)}^p \leq N\,\|T\|_{L^p(\mathbb{T}^d)}^p,\]
where $N$ is the number of lattice points in $\lceil p/2\rceil$ times the convex hull of the spectrum of $T$. Suppose now that $2<p<\infty$. If the spectrum is contained in a box, then Corollary~\ref{cor:nikolskii}, applied in one variable at a time, shows that $\lceil p/2 \rceil$ may be replaced by $p/2$ in the resulting formula for $N$. If the spectrum is contained in the simplex $\{\alpha \in \mathbb{N}_0^d : \lvert \alpha \rvert \leq n\}$, then $N = \binom{d+n\lceil p/2 \rceil}{d}$. Using Theorem~\ref{thm:dgen} (with $d+1$ in place of $d$) as in the proof of Corollary~\ref{cor:nikolskii}, together with \eqref{eq:binom}, yields the same improvement in this case.

\subsection{Khinchin's inequality for Steinhaus variables}
If $Z_1,\ldots,Z_d$ are independent random variables, each uniformly distributed on $\mathbb{T}$, then 
\[\mathbb{E}\left\lvert \sum_{j=1}^d a_j Z_j \right\rvert^p = \left\|\sum_{j=1}^d a_j z_j\right\|_{H^p(\mathbb{D}^d)}^p.\]
The sharp Khinchin inequality for Steinhaus variables, due independently to Baernstein--Culverhouse~\cite{BC2002}*{Corollary~2} and K\"onig--Kwapie\'n~\cite{KK2001}*{Theorem~3}, asserts that if $2 < p < \infty$, then 
\[\left\|\sum_{j=1}^d a_j z_j\right\|_{H^p(\mathbb{D}^d)} \leq \Gamma(1+p/2)^{1/p} \|a\|_{\ell^2},\]
and the constant is optimal as $d\to\infty$. For fixed $d$, the optimal constant is attained when all coefficients are equal, so that
\begin{equation} \label{eq:Bpd}
    B_{p,d} \coloneq \sup_{\|a\|_{\ell^2}=1} \left\|\sum_{j=1}^d a_j z_j\right\|_{H^p(\mathbb{D}^d)} = \frac{\|z_1+\cdots+z_d\|_{H^p(\mathbb{D}^d)}}{\sqrt{d}}.
\end{equation}
This follows from~\cite{BC2002}*{Theorem~1}, which moreover shows that $B_{p,d}$ is increasing in $d$ (see also~\cite{KK2001}*{p.~122}). A different proof of the identity \eqref{eq:Bpd} is given in the proof of~\cite{BOCS2023}*{Theorem~5.2}, and~\cite{BOCS2023}*{Lemma~5.1} shows that $B_{p,d}$ is strictly increasing in $d$. In particular, $B_{p,d}^p<\Gamma(1+p/2)$ for every $d$, since $B_{p,d}^p \to \Gamma(1+p/2)$ as $d\to\infty$ by \eqref{eq:Bpd} and the central limit theorem. Theorem~\ref{thm:dgen} supplies a lower bound.

\begin{corollary} \label{cor:khinchin}
	If $2<p<\infty$ and $d\geq2$, then 
	\[\frac{d^{p/2}\,\Gamma(d)}{\Gamma(d+p/2)} < \frac{B_{p,d}^p}{\Gamma(1+p/2)}<1.\]
\end{corollary}

\begin{proof}
	The upper bound follows from~\cite{BOCS2023}*{Lemma~5.1}, as explained above. For the lower bound, we apply Theorem~\ref{thm:dgen} to $f(z) \coloneq z_1+\cdots+z_d$, for which $\|f\|_{H^p(\mathbb{D}^d)}^p = d^{p/2}B_{p,d}^p$ and $D_d f(w) = dw$, so that
	\[\|D_d f\|_{A^p_d(\mathbb{D})}^p = \frac{d^p\,\Gamma(d)\,\Gamma(1+p/2)}{\Gamma(d+p/2)}\]
	by \eqref{eq:binom} with $s=p/2$. The inequality is strict, since $f$ vanishes at the origin and is therefore not one of the extremals in Theorem~\ref{thm:dgen}.
\end{proof}
The lower bound in Corollary~\ref{cor:khinchin} can also be obtained by averaging 
\[a \mapsto \left\|\sum_{j=1}^d a_j z_j\right\|_{H^p(\mathbb{D}^d)}^p\] over the unit sphere of $\mathbb{C}^d$. The two arguments give the same bound because $\mu_d$ is the distribution of $a_1$ when $a$ is uniformly distributed on the sphere (see~\cite{BBHOCP2019}*{Lemma~7}).

By Stirling's formula, the lower bound in Corollary~\ref{cor:khinchin} satisfies
\begin{align*}
    \frac{d^{p/2}\,\Gamma(d)}{\Gamma(d+p/2)} &= 1-\frac{p(p-2)}{8d}+O(d^{-2})
    \intertext{as $d\to\infty$. If $p$ is an even integer, then}
    \frac{B_{p,d}^p}{\Gamma(1+p/2)} &= 1-\frac{p(p-2)}{16d}+O(d^{-2}).
\end{align*}
For the second expansion, see~\cite{BNSW2011}*{Remark~2}, where it is also explained how it formally extends to general $p$.

\begin{remark}
	The proof of Khinchin's inequality in~\cite{BOCS2023}*{Section~5} relies on an explicit formula~\cite{BOCS2023}*{Theorem~1.2} for the Riesz projection of $\lvert\varphi\rvert^{p-2}\varphi$ when $\varphi(z)=\sum_{j=1}^d c_j z_j$. This formula is established in~\cite{BOCS2023} by a computation for even integers $p$ followed by an approximation argument. In the spirit of the present paper, we note that for $2 \leq p < \infty$ it is also a direct consequence of Green's formula. Indeed, the Riesz projection of $\lvert\varphi\rvert^{p-2}\varphi$ is a linear polynomial, whose coefficient of $z_j$ is
	\[\int_{\mathbb{T}^d} \lvert\varphi(\zeta)\rvert^{p-2}\varphi(\zeta)\,\overline{\zeta_j}\,dm_d(\zeta).\]
	Applying Green's formula in $\zeta_j$ and using that $\partial_j(\lvert\varphi\rvert^{p-2}\varphi) = \frac{p}{2}\,c_j\,\lvert\varphi\rvert^{p-2}$, we find that this coefficient equals
	\[\frac{p}{2}\,c_j \int_{\mathbb{T}^d}\int_{\mathbb{D}} \lvert\varphi(\zeta_1,\ldots,\zeta_{j-1},w,\zeta_{j+1},\ldots,\zeta_d)\rvert^{p-2}\,d\mu(w)\,dm_d(\zeta),\]
	which is the formula of~\cite{BOCS2023}*{Theorem~1.2}.
\end{remark}

\subsection{The adjoint of \texorpdfstring{$D$}{D} and an \texorpdfstring{$L^p$}{Lp} Douglas inequality}
For an analytic function $g$ in $\mathbb{D}$, let
\begin{equation} \label{eq:adjoint}
    D^\ast g(z) \coloneq \int_0^1 g((1-t)z_1+tz_2)\,dt,
\end{equation}
so that $D^\ast a' = a[z_1,z_2]$. The restriction of $D^\ast$ to $A^2(\mathbb{D})$ is the adjoint of $D$ from $H^2(\mathbb{D}^2)$ to $A^2(\mathbb{D})$. Since $D$ is bounded from $H^p(\mathbb{D}^2)$ to $A^p(\mathbb{D})$ for $0<p<\infty$ by \cites{DS1975,HO1975} and the Riesz projection is bounded on $L^p(\mathbb{T}^2)$ for $1<p<\infty$, a duality argument shows that
\[g \in A^p(\mathbb{D}) \qquad\iff\qquad D^\ast g \in H^p(\mathbb{D}^2)\]
for $1<p<\infty$. Since the boundary values of $D^\ast a'$ are $(a(\zeta_1)-a(\zeta_2))/(\zeta_1-\zeta_2)$, this equivalence is the familiar description of an analytic Besov space in terms of the derivative (see e.g.~\cite{Dyakonov1998}*{p.~143} and the references there).

If $2<p<\infty$, then Theorem~\ref{thm:d2} and the identity $g = DD^\ast g$ show that $\|g\|_{A^p(\mathbb{D})} \leq \|D^\ast g\|_{H^p(\mathbb{D}^2)}$. Writing $g=a'$, we may state this as follows.

\begin{corollary} \label{cor:boundary}
	Suppose that $2<p<\infty$ and that $a'$ is in $A^p(\mathbb{D})$. Then
	\begin{equation} \label{eq:boundary}
		\int_{\mathbb{D}} \lvert a'(w)\rvert^p\,d\mu(w) \leq \int_{\mathbb{T}^2} \left\lvert \frac{a(\zeta_1)-a(\zeta_2)}{\zeta_1-\zeta_2}\right\rvert^p dm_2(\zeta),
	\end{equation}
	and equality holds in this bound if and only if $a$ is affine.
\end{corollary}

\begin{proof}
	The inequality was established above. If $a$ is affine, then $D^\ast a'$ is constant and equality holds trivially. Conversely, suppose that equality holds and that $a$ is not constant, and set $f \coloneq D^\ast a'$ and $q \coloneq p/2$. By Theorem~\ref{thm:d2}, $f$ does not vanish in $\mathbb{D}^2$ and $f^q$ is a divided difference, so $\mathcal{L}(f^q) \equiv 0$. Since $f$ is itself a divided difference, we also have $\mathcal{L}f\equiv0$, and the identity
	\[\mathcal{L}(f^q) = q f^{q-2}(f\,\mathcal{L}f + (q-1)(z_1-z_2)\,\partial_1 f\,\partial_2 f)\]
	shows that $\partial_1 f\,\partial_2 f \equiv 0$. As $f$ is symmetric, this means that $\partial_1 f \equiv 0$, so $f$ is constant and $a$ is affine.
\end{proof}

If $g(w) = 2w$, then $D^\ast g(z) = z_1+z_2$, and the computation in the introduction shows that \eqref{eq:boundary} fails for every $0<p<2$. If $p=2$, then Theorem~\ref{thm:rudin} shows that $D^\ast$ is an isometry from $A^2(\mathbb{D})$ to $H^2(\mathbb{D}^2)$, so equality holds in \eqref{eq:boundary} for every $a$ with $a'$ in $A^2(\mathbb{D})$. This is the classical Douglas formula~\cite{Douglas1931} (see also~\cite{EKMR2014}*{Theorem~1.5.1}).

\begin{remark}
   Since $DD^\ast$ is the identity, $D^\ast D$ is the orthogonal projection of $H^2(\mathbb{D}^2)$ onto the subspace of divided differences. It is plain from the definition of $D^\ast$ in \eqref{eq:adjoint} that $D^\ast D$ is also contractive on $H^\infty(\mathbb{D}^2)$. However, it is not contractive on $H^p(\mathbb{D}^2)$ for any $2<p<\infty$. Indeed, if $f$ attains equality in \eqref{eq:mainbound} and $Df$ is not constant, for example $f(z) = (2+z_1+z_2)^{2/p}$, then Corollary~\ref{cor:boundary} yields $\|D^\ast D f\|_{H^p(\mathbb{D}^2)} > \|Df\|_{A^p(\mathbb{D})} = \|f\|_{H^p(\mathbb{D}^2)}$. 
\end{remark}

\bibliography{rudindiag}

\end{document}